\documentclass[11pt]{article}
\usepackage{amsmath}
\usepackage{amssymb}
\usepackage{amsthm}
\usepackage{amscd}
\usepackage{amsfonts}
\usepackage{graphicx}
\usepackage{fancyhdr}
\usepackage{tikz}
\usepackage{authblk}

\usepackage[numbers]{natbib}

\usetikzlibrary{cd}
\usetikzlibrary{positioning,decorations.pathmorphing}
\usetikzlibrary{arrows.meta}

\theoremstyle{plain} \numberwithin{equation}{section}
\newtheorem{theorem}{Theorem}[section]
\newtheorem{corollary}[theorem]{Corollary}

\newtheorem{lemma}[theorem]{Lemma}

\theoremstyle{definition}
\newtheorem{defn}[theorem]{Definition}

\newtheorem{remark}[theorem]{Remark}

\newtheorem{ex}[theorem]{Example}

\newtheorem*{theorem*}{Theorem}

\title{A Broken Circuit Model for Chromatic Homology Theories}

\author[1]{Alex Chandler}
\author[2]{Radmila Sazdanovi\'{c}}
\affil[1]{Faculty of Mathematics, University of Vienna}
\affil[2]{Department of Mathematics, North Carolina State University}

\newcommand{\Tot}{\textnormal{Tot}}

\newcommand{\rk}{\textnormal{rk}}
\newcommand{\qrk}{q\textnormal{rk}}
 
\newcommand{\Ob}{\textnormal{Ob}}

\newcommand{\CSNgmod}{\C[S_n]\textbf{-gmod}}
\newcommand{\CSMgmod}{\C[S_m]\textbf{-gmod}}
\newcommand{\Ind}{\textnormal{Ind}}
\newcommand{\BCC}{\textnormal{BC}}
\newcommand{\mc}[1]{\mathcal{#1}}

\newcommand{\Z}{\mathbb{Z}}
\newcommand{\N}{\mathbb{N}}

\newcommand{\C}{\mathbb{C}}
\newcommand{\MBC}{\textnormal{M}_{\textnormal{BC}}} 
\newcommand{\Ch}{\textnormal{Ch}} 
\newcommand{\ch}{\textnormal{ch}} 
\newcommand{\FCH}{F^{\textnormal{Ch}}} 
\newcommand{\HCH}{H^{\textnormal{Ch}}} 
\newcommand{\FCS}{F^{\textnormal{CS}}} 
\newcommand{\HCS}{H^{\textnormal{CS}}}

\newcommand{\NBC}{\textnormal{NBC}}

\newcommand{\Rmod}{R\textnormal{-}\mathbf{mod}}

\newcommand{\Rgmod}{R\textnormal{-}\mathbf{gmod}}

\definecolor{forest}{rgb}{0.03, 0.47, 0.19}

\newsavebox\si
\begin{lrbox}{\si}
\begin{tikzpicture}
\draw (0,0) circle (2cm);
\draw[ultra thick] (-1,0.5)--node[below]{$e_j$}(1,0.5);
\draw[ultra thick] (-1,0.5) to [out=90,in=90] (1,0.5);
\node at (0,1.5){$C_j$};
\draw[ultra thick,dotted] (-1,-1)--node[below]{$x$}(1,-1);
\draw[ultra thick,dotted] (-1,-.5)--node[below]{$e_i$}(1,-.5);
\draw (-1,-.5) to[out=135,in=-135](-1,.5);
\draw (1,-.5) to[out=45,in=-45](1,.5);
\node at (-1.5,0){$C_i$};
\end{tikzpicture}
\end{lrbox}

\newsavebox\sj
\begin{lrbox}{\sj}
\begin{tikzpicture}
\draw (0,0) circle (2cm);
\draw[ultra thick,dotted] (-1,0.5)--node[below]{$e_j$}(1,0.5);
\draw[ultra thick] (-1,0.5) to [out=90,in=90] (1,0.5);
\node at (0,1.5){$C_j$};
\draw[ultra thick] (-1,-1)--node[below]{$x$}(1,-1);
\draw[ultra thick,dotted] (-1,-.5)--node[below]{$e_i$}(1,-.5);
\draw (-1,-.5) to[out=135,in=-135](-1,.5);
\draw (1,-.5) to[out=45,in=-45](1,.5);
\node at (-1.5,0){$C_i$};
\end{tikzpicture}
\end{lrbox}

\newsavebox\tj
\begin{lrbox}{\tj}
\begin{tikzpicture}
\draw (0,0) circle (2cm);
\draw[ultra thick] (-1,0.5)--node[below]{$e_j$}(1,0.5);
\draw[ultra thick] (-1,0.5) to [out=90,in=90] (1,0.5);
\node at (0,1.5){$C_j$};
\draw[ultra thick] (-1,-1)--node[below]{$x$}(1,-1);
\draw[ultra thick,dotted] (-1,-.5)--node[below]{$e_i$}(1,-.5);
\draw (-1,-.5) to[out=135,in=-135](-1,.5);
\draw (1,-.5) to[out=45,in=-45](1,.5);
\node at (-1.5,0){$C_i$};
\end{tikzpicture}
\end{lrbox}

\newsavebox\graph
\begin{lrbox}{\graph}
\begin{tikzpicture}[ every node/.style={inner sep=0,outer sep=-1}]
\node (a) at (1,0){\textbullet};
\node (b) at (0,1){\textbullet};
\node (c) at (-1,0){\textbullet};
\node (d) at (0,-1){\textbullet};
\node at (-.7,.7){2};
\node at (.7,.7){1};
\node at (-.7,-.7){3};
\node at (.7,-.7){4};
\node at (0,.2){5};
\draw[ very thick] (a)--(b)--(c)--(d)--(a);
\draw[ very thick] (a)--(c);
\end{tikzpicture}
\end{lrbox}

\newsavebox\graphonetwo
\begin{lrbox}{\graphonetwo}
\begin{tikzpicture}[scale=.5, every node/.style={ scale=.7, inner sep=0,outer sep=-1}]
\node (a) at (1,0){\scalebox{1.5}{\textbullet}};
\node (b) at (0,1){\scalebox{1.5}{\textbullet}};
\node (c) at (-1,0){\scalebox{1.5}{\textbullet}};
\node (d) at (0,-1){\scalebox{1.5}{\textbullet}};
\draw[ultra thick] (a)--(b);
\draw[ultra thick] (b)--(c);
\end{tikzpicture}
\end{lrbox}

\newsavebox\graphonetwofour
\begin{lrbox}{\graphonetwofour}
\begin{tikzpicture}[scale=.5, every node/.style={ scale=.7, inner sep=0,outer sep=-1}]
\node (a) at (1,0){\scalebox{1.5}{\textbullet}};
\node (b) at (0,1){\scalebox{1.5}{\textbullet}};
\node (c) at (-1,0){\scalebox{1.5}{\textbullet}};
\node (d) at (0,-1){\scalebox{1.5}{\textbullet}};
\draw[ultra thick] (a)--(b);
\draw[ultra thick] (b)--(c);
\draw[ultra thick] (a)--(d);
\end{tikzpicture}
\end{lrbox}

\newsavebox\graphonetwothreefour
\begin{lrbox}{\graphonetwothreefour}
\begin{tikzpicture}[scale=.5, every node/.style={ scale=.7, inner sep=0,outer sep=-1}]
\node (a) at (1,0){\scalebox{1.5}{\textbullet}};
\node (b) at (0,1){\scalebox{1.5}{\textbullet}};
\node (c) at (-1,0){\scalebox{1.5}{\textbullet}};
\node (d) at (0,-1){\scalebox{1.5}{\textbullet}};
\draw[ultra thick] (a)--(b);
\draw[ultra thick] (b)--(c);
\draw[ultra thick] (c)--(d);
\draw[ultra thick] (a)--(d);
\end{tikzpicture}
\end{lrbox}

\newsavebox\graphonetwothreefive
\begin{lrbox}{\graphonetwothreefive}
\begin{tikzpicture}[scale=.5, every node/.style={ scale=.7, inner sep=0,outer sep=-1}]
\node (a) at (1,0){\scalebox{1.5}{\textbullet}};
\node (b) at (0,1){\scalebox{1.5}{\textbullet}};
\node (c) at (-1,0){\scalebox{1.5}{\textbullet}};
\node (d) at (0,-1){\scalebox{1.5}{\textbullet}};
\draw[ultra thick] (a)--(b);
\draw[ultra thick] (b)--(c);
\draw[ultra thick] (c)--(d);
\draw[ultra thick] (a)--(c);
\end{tikzpicture}
\end{lrbox}

\newsavebox\graphonetwofourfive
\begin{lrbox}{\graphonetwofourfive}
\begin{tikzpicture}[scale=.5, every node/.style={ scale=.7, inner sep=0,outer sep=-1}]
\node (a) at (1,0){\scalebox{1.5}{\textbullet}};
\node (b) at (0,1){\scalebox{1.5}{\textbullet}};
\node (c) at (-1,0){\scalebox{1.5}{\textbullet}};
\node (d) at (0,-1){\scalebox{1.5}{\textbullet}};
\draw[ultra thick] (a)--(b);
\draw[ultra thick] (b)--(c);
\draw[ultra thick] (a)--(d);
\draw[ultra thick] (a)--(c);
\end{tikzpicture}
\end{lrbox}

\newsavebox\graphonetwothreefourfive
\begin{lrbox}{\graphonetwothreefourfive}
\begin{tikzpicture}[scale=.5, every node/.style={ scale=.7, inner sep=0,outer sep=-1}]
\node (a) at (1,0){\scalebox{1.5}{\textbullet}};
\node (b) at (0,1){\scalebox{1.5}{\textbullet}};
\node (c) at (-1,0){\scalebox{1.5}{\textbullet}};
\node (d) at (0,-1){\scalebox{1.5}{\textbullet}};
\draw[ultra thick] (a)--(b);
\draw[ultra thick] (b)--(c);
\draw[ultra thick] (c)--(d);
\draw[ultra thick] (a)--(d);
\draw[ultra thick] (a)--(c);
\end{tikzpicture}
\end{lrbox}

\newsavebox\graphonetwothree
\begin{lrbox}{\graphonetwothree}
\begin{tikzpicture}[scale=.5, every node/.style={ scale=.7, inner sep=0,outer sep=-1}]
\node (a) at (1,0){\scalebox{1.5}{\textbullet}};
\node (b) at (0,1){\scalebox{1.5}{\textbullet}};
\node (c) at (-1,0){\scalebox{1.5}{\textbullet}};
\node (d) at (0,-1){\scalebox{1.5}{\textbullet}};
\draw[ultra thick] (a)--(b);
\draw[ultra thick] (b)--(c);
\draw[ultra thick] (c)--(d);
\end{tikzpicture}
\end{lrbox}

\newsavebox\graphonetwofive
\begin{lrbox}{\graphonetwofive}
\begin{tikzpicture}[scale=.5, every node/.style={ scale=.7, inner sep=0,outer sep=-1}]
\node (a) at (1,0){\scalebox{1.5}{\textbullet}};
\node (b) at (0,1){\scalebox{1.5}{\textbullet}};
\node (c) at (-1,0){\scalebox{1.5}{\textbullet}};
\node (d) at (0,-1){\scalebox{1.5}{\textbullet}};
\draw[ultra thick] (a)--(b);
\draw[ultra thick] (b)--(c);
\draw[ultra thick] (a)--(c);
\end{tikzpicture}
\end{lrbox}

\newsavebox\graphthreefour
\begin{lrbox}{\graphthreefour}
\begin{tikzpicture}[scale=.5, every node/.style={ scale=.7, inner sep=0,outer sep=-1}]
\node (a) at (1,0){\scalebox{1.5}{\textbullet}};
\node (b) at (0,1){\scalebox{1.5}{\textbullet}};
\node (c) at (-1,0){\scalebox{1.5}{\textbullet}};
\node (d) at (0,-1){\scalebox{1.5}{\textbullet}};
\draw[ultra thick] (c)--(d);
\draw[ultra thick] (d)--(a);
\end{tikzpicture}
\end{lrbox}

\newsavebox\graphtwothreefour
\begin{lrbox}{\graphtwothreefour}
\begin{tikzpicture}[scale=.5, every node/.style={ scale=.7, inner sep=0,outer sep=-1}]
\node (a) at (1,0){\scalebox{1.5}{\textbullet}};
\node (b) at (0,1){\scalebox{1.5}{\textbullet}};
\node (c) at (-1,0){\scalebox{1.5}{\textbullet}};
\node (d) at (0,-1){\scalebox{1.5}{\textbullet}};
\draw[ultra thick] (b)--(c);
\draw[ultra thick] (c)--(d);
\draw[ultra thick] (d)--(a);
\end{tikzpicture}
\end{lrbox}

\newsavebox\graphonethreefour
\begin{lrbox}{\graphonethreefour}
\begin{tikzpicture}[scale=.5, every node/.style={ scale=.7, inner sep=0,outer sep=-1}]
\node (a) at (1,0){\scalebox{1.5}{\textbullet}};
\node (b) at (0,1){\scalebox{1.5}{\textbullet}};
\node (c) at (-1,0){\scalebox{1.5}{\textbullet}};
\node (d) at (0,-1){\scalebox{1.5}{\textbullet}};
\draw[ultra thick] (a)--(b);
\draw[ultra thick] (c)--(d);
\draw[ultra thick] (d)--(a);
\end{tikzpicture}
\end{lrbox}

\newsavebox\graphthreefourfive
\begin{lrbox}{\graphthreefourfive}
\begin{tikzpicture}[scale=.5, every node/.style={ scale=.7, inner sep=0,outer sep=-1}]
\node (a) at (1,0){\scalebox{1.5}{\textbullet}};
\node (b) at (0,1){\scalebox{1.5}{\textbullet}};
\node (c) at (-1,0){\scalebox{1.5}{\textbullet}};
\node (d) at (0,-1){\scalebox{1.5}{\textbullet}};
\draw[ultra thick] (c)--(d);
\draw[ultra thick] (d)--(a);
\draw[ultra thick] (c)--(a);
\end{tikzpicture}
\end{lrbox}

\newsavebox\graphtwothreefourfive
\begin{lrbox}{\graphtwothreefourfive}
\begin{tikzpicture}[scale=.5, every node/.style={ scale=.7, inner sep=0,outer sep=-1}]
\node (a) at (1,0){\scalebox{1.5}{\textbullet}};
\node (b) at (0,1){\scalebox{1.5}{\textbullet}};
\node (c) at (-1,0){\scalebox{1.5}{\textbullet}};
\node (d) at (0,-1){\scalebox{1.5}{\textbullet}};
\draw[ultra thick] (b)--(c);
\draw[ultra thick] (c)--(d);
\draw[ultra thick] (d)--(a);
\draw[ultra thick] (c)--(a);
\end{tikzpicture}
\end{lrbox}

\newsavebox\graphonethreefourfive
\begin{lrbox}{\graphonethreefourfive}
\begin{tikzpicture}[scale=.5, every node/.style={ scale=.7, inner sep=0,outer sep=-1}]
\node (a) at (1,0){\scalebox{1.5}{\textbullet}};
\node (b) at (0,1){\scalebox{1.5}{\textbullet}};
\node (c) at (-1,0){\scalebox{1.5}{\textbullet}};
\node (d) at (0,-1){\scalebox{1.5}{\textbullet}};
\draw[ultra thick] (a)--(b);
\draw[ultra thick] (c)--(d);
\draw[ultra thick] (d)--(a);
\draw[ultra thick] (c)--(a);
\end{tikzpicture}
\end{lrbox}

\newsavebox\graphonethreefive
\begin{lrbox}{\graphonethreefive}
\begin{tikzpicture}[scale=.5, every node/.style={ scale=.7, inner sep=0,outer sep=-1}]
\node (a) at (1,0){\scalebox{1.5}{\textbullet}};
\node (b) at (0,1){\scalebox{1.5}{\textbullet}};
\node (c) at (-1,0){\scalebox{1.5}{\textbullet}};
\node (d) at (0,-1){\scalebox{1.5}{\textbullet}};
\draw[ultra thick] (a)--(b);
\draw[ultra thick] (c)--(d);
\draw[ultra thick] (c)--(a);
\end{tikzpicture}
\end{lrbox}

\newsavebox\linearextension
\begin{lrbox}{\linearextension}
\raisebox{-2mm}{\scalebox{.5}{\usebox\graphthreefour}} 
$\leq $
\raisebox{-2mm}{\scalebox{.5}{\usebox\graphthreefourfive}} 
$\leq $
\raisebox{-2mm}{\scalebox{.5}{\usebox\graphonetwo}} 
$\leq $
\raisebox{-2mm}{\scalebox{.5}{\usebox\graphonetwofive}} 
$\leq $
\raisebox{-2mm}{\scalebox{.5}{\usebox\graphtwothreefour}} 
$\leq $
\raisebox{-2mm}{\scalebox{.5}{\usebox\graphtwothreefourfive}} 
$\leq $
\raisebox{-2mm}{\scalebox{.5}{\usebox\graphonethreefour}} 
$\leq $
\raisebox{-2mm}{\scalebox{.5}{\usebox\graphonethreefourfive}} 
$\leq $
\raisebox{-2mm}{\scalebox{.5}{\usebox\graphonetwofour}} 
$\leq $
\raisebox{-2mm}{\scalebox{.5}{\usebox\graphonetwofourfive}} 
$\leq $
\raisebox{-2mm}{\scalebox{.5}{\usebox\graphonetwothree}} 
$\leq $
\raisebox{-2mm}{\scalebox{.5}{\usebox\graphonetwothreefive}} 
$\leq $
\raisebox{-2mm}{\scalebox{.5}{\usebox\graphonetwothreefourfive}}
\end{lrbox}

\newsavebox\bcexample
\begin{lrbox}
{\bcexample}
\begin{tikzpicture}
\node (otw) at (0,5){\usebox\graphonetwo};

\node (thfo) at (0,7){\usebox\graphthreefour};
\node (otwth)at (4,0){\usebox\graphonetwothree};

\node (otwfo)at (4,2){\usebox\graphonetwofour};

\node (otwfi)at (4,4){\usebox\graphonetwofive};

\node (othfo)at (4,6){\usebox\graphonethreefour};

\node (twthfo)at (4,8){\usebox\graphtwothreefour};

\node (thfofi)at (4,10){\usebox\graphthreefourfive};

\node (otwthfo)at (8,1){\usebox\graphonetwothreefour};

\node (otwthfi)at (8,3){\usebox\graphonetwothreefive};

\node (otwfofi)at (8,5){\usebox\graphonetwofourfive};

\node (othfofi)at (8,7){\usebox\graphonethreefourfive};

\node (twthfofi)at (8,9){\usebox\graphtwothreefourfive};

\node (otwthfofi)at (12,5){\usebox\graphonetwothreefourfive};

\draw[->,blue,very thick] (otw)--(otwth);
\draw[->,blue,very thick] (otw)--(otwfo);
\draw[<-,red,very thick] (otw)--(otwfi);

\draw[->,blue,very thick] (thfo)--(othfo);
\draw[->,blue,very thick] (thfo)--(twthfo);
\draw[<-,red,very thick] (thfo)--(thfofi);

\draw[->,blue,very thick] (otwth)--(otwthfo);
\draw[<-,red,very thick] (otwth)--(otwthfi);

\draw[->,blue,very thick] (otwfo)--(otwthfo);
\draw[<-,red,very thick] (otwfo)--(otwfofi);

\draw[->,blue,very thick] (otwfi)--(otwthfi);
\draw[->,blue,very thick] (otwfi)--(otwfofi);

\draw[->,blue,very thick] (othfo)--(otwthfo);
\draw[<-,red,very thick] (othfo)--(othfofi);

\draw[->,blue,very thick] (twthfo)--(otwthfo);
\draw[<-,red,very thick] (twthfo)--(twthfofi);

\draw[->,blue,very thick] (thfofi)--(othfofi);
\draw[->,blue,very thick] (thfofi)--(twthfofi);

\draw[<-,red,very thick] (otwthfo)--(otwthfofi);

\draw[->,blue,very thick] (otwthfi)--(otwthfofi);

\draw[->,blue,very thick] (otwfofi)--(otwthfofi);

\draw[->,blue,very thick] (othfofi)--(otwthfofi);

\draw[->,blue,very thick] (twthfofi)--(otwthfofi);
\end{tikzpicture}
\end{lrbox}

\newtheorem{case}{Case}
\usepackage{wrapfig,comment}

\begin{document}

\maketitle

\begin{abstract}
Using the tools of algebraic Morse theory, and the thin poset approach to constructing homology theories, we give a categorification of Whitney's broken circuit theorem for the chromatic polynomial, and for Stanley's chromatic symmetric function.
%This approach provides bounds on the homological span 
% As an application, we are able to obtain bounds on the homological span 
%of chromatic homology, chromatic symmetric homology, and cohomology rings of graph configuration spaces.
\end{abstract}

%\tableofcontents

\section{Introduction}

% In graph theory, as in any field of mathematics, one is interested in classification. Classification problems are often made more tractable by the use of \textit{invariants}. Thus graph theorists are often interested in defining and studying \textit{graph invariants}, that is, functions $\chi:\{\textnormal{graphs}\}\to S$ where $S$ is a set, such that if $G$ and $G^\prime$ are isomorphic graphs, then $\chi(G)=\chi(G^\prime)$ in $S$. Such functions provide a tool for distinguishing graphs: if $\chi(G)\neq \chi(G^\prime)$ then $G$ and $G^\prime$ are not isomorphic. 

In his groundbreaking work at the beginning of the century, Khovanov \cite{khovanov2000categ} categorified the Jones polynomial, a powerful link invariant, lifting it to a bigraded homology theory.  Categorification has since found applications in many areas of mathematics, including graph theory \cite{everitt2009homology, turnereverittcellular, helme2005categorification, hepworth2015categorifying, jasso2006tutte, loebl2008, sazdanovic2018categorification, stovsic2008categorification}. 
The study of graph invariants largely began with the definition of the chromatic polynomial by Birkhoff \cite{birkhoff1912determinant} in 1912, which led to generalizations such as the Tutte polynomial \cite{tutte54}, and the chromatic symmetric function \cite{stanley1995symmetric}. Along with the Tutte polynomial, came a connection between graphs and links in the 3-sphere. In particular, the Tutte polynomial of a graph specializes to the Jones polynomial of an associated alternating link.

% bigraded abelian group which is obtained as the cohomology of a bigraded chain complex. One recovers the Jones polynomial from this bigraded abelian group by taking its graded Euler characteristic.
% The huge success of this theory encouraged others upgrade their favorite polynomial invariants to homological invariants via the process of categorification. 
% In this paper, we consider two such homological graph invariants: 
This work deals with two categorifications of the chromatic polynomial: a Khovanov-type chromatic homology defined by 
Helme-Guizon and Rong \cite{helme2005categorification}, and  Eastwood and Hugget's chromatic homology \cite{eastwood2007euler}, as well as the chromatic symmetric homology by Sazdanovi\'{c} and Yip \cite{sazdanovic2018categorification} that categorifies Stanley's chromatic symmetric function \cite{stanley1995symmetric}. 
% In the process of categorification, properties of polynomials often get lifted to more refined properties at the categorified level, and the present context is no exception. 
One of the goals of categorifying invariants is to obtain a richer theory and more refined invariants. As a great side effect, properties of the original invariants, polynomials in our case, should be present on the categorified level, as a more refined, structural or computational tool. For example, the defining deletion-contraction relation corresponds to a short exact sequence of chain groups, and therefore, long exact sequence in (co)homology in the work of Helme-Guizon, Rong and Eastwood and Huggett. 

The main result of this paper builds on the thin poset approach to Khovanov-like homology theories developed by the first author in \cite{chandlerthesis, chandlerthin}, along with the tools of algebraic Morse theory \cite{forman2002user,  jollenbeck2005resolution, skoldberg2006morse}, to categorify Whitney's broken circuit model for the chromatic polynomial \cite{whitney1932logical}, and the chromatic symmetric function \cite{stanley1995symmetric}.  This work provides the right setting for extending the work of Baranovsky and Zubkov \cite{baranovsky2017chromatic} on generalizing Helme-Guizon and Rong chromatic homology to brace algebras from trees to all graphs. In general, it could provide ideas for computing cohomology of configuration  spaces  and their generalizations \cite{bendersky1991cohomology, cooper2019configuration, kriz,totaro} over $\mathbb{Z}$ and $\mathbb{Z}_p.$

This paper is organized as follows. In Section \ref{pocohom}, we give background information on how to obtain cohomology theories from functors on posets as described by the first author in \cite{chandlerthesis, chandlerthin}. 
In Section \ref{sectionmorse}, we provide a brief background on algebraic Morse theory, and give an application to the cohomology theories constructed from functors on posets in Theorem \ref{morsethinposet}. 
In Section \ref{whitney}, we prove Lemma \ref{bcmatching}, where we establish an acyclic matching on the Hasse diagram of the Boolean lattice of spanning subgraphs of a given graph, matching pairs of subgraphs containing broken circuits which have the same vertex partition in terms of connected components. In Section \ref{brokenchromatic}, we describe Helme-Guizon and Rong's chromatic homology in terms of the cohomology of a functor on a poset. An application of Theorem \ref{morsethinposet},  together with the Lemma \ref{bcmatching}, gives Theorem \ref{categcwhitney}, a categorification of Whitney's broken circuit theorem for chromatic polynomials. We are able to use this result to obtain bounds on the homological span for chromatic homology. In Section \ref{brokensymmetric}, we prove the analogous results for the chromatic symmetric homology of Sazdanovi\'{c} and Yip. Finally in Section \ref{spectralsequencesec}, we show that the first page in Sazdanovi\'{c} and Baranovsky's spectral sequence from chromatic homology to the cohomology of a graph configuration space can be replaced by a broken circuit model, thus yielding bounds on the homological span for cohomology rings of graph configuration spaces.
 \section*{Acknowledgments} The authors are thankful to M. Yip for many useful discussions especially on the preliminary ideas for categorifying Whitney's Broken circuit theorem. Many thanks to V. Baranovsky for the invaluable insight into generalizing chromatic homology to non-commutative algebras. 
 RS was partially supported by the Simons Collaboration Grant 318086.

\section{Cohomology Theories from Thin Posets}
\label{pocohom}

% We begin with the necessary definitions and background on the theory of partially ordered sets. If needed, the reader can refer to \cite{stanley1998enumerative} for anything not mentioned here.
This section provides a background on the theory of partially ordered sets. For more details see \cite{stanley1998enumerative}.
A \textit{partially ordered set (poset)} is a set $P$ together with a reflexive, antisymmetric, transitive relation, denoted $\leq$. A \textit{cover relation} in a poset is a pair $(x,y)$ with $x\leq y$ such that there is no $z\in P$ with $x<z<y$, and in this case we write $x\lessdot y$. Let $C(P)$ denote the set of all cover relations in $P$. The \textit{Hasse diagram} of a finite poset is the directed graph $(P,C(P))$ with nodes $P$ and a directed edge drawn from left to right from $x$ to $y$ if and only if $x\lessdot y$. A poset $P$ is \textit{graded} if there is a function $\rk:P\to\N$ so that $x\lessdot y$ implies $\rk(y)=\rk(x)+1$. A  graded poset is \textit{thin} if every nonempty closed interval $[x,y]=\{z\in P \ | \ x\leq z\leq y\}$ with $\rk(y)=\rk(x)+2$ consists of exactly 4 elements. In a thin poset, such an interval is called a \textit{diamond}. A \textit{balanced coloring} of a thin poset is a function $c:C(P)\to \{+1,-1\}$ such that each diamond in $P$ has an odd number of cover relations colored by $-1$'s.

A poset $P$ can be thought of equivalently as a category with objects $P$ and with a unique morphism from $x$ to $y$ if and only if $x\leq y$ in $P$. A functor on a poset is equivalent to a labeling of the nodes and edges of the Hasse diagram of the poset so that for any $x\leq y$, compositions along any two directed paths between $x$ and $y$ agree. A poset $P$ can also be thought of as a topological space with underlying set $P$ and open sets consisting of the lower order ideals in $P$. Functors on $P$ can be identified with presheaves in this viewpoint, the stalk of a sheaf at $x\in P$ corresponding to the functor value at $x$. In this section, we will construct a cohomology theory $H^*(F)$ associated to a functor $F:P\to\mc{A}$ where $P$ is a thin poset. For more details on this construction, see \cite{chandlerthin}. More generally, given any poset $P$, one can use the presheaf viewpoint to define a cohomology theory built from the right derived functors of the left exact global section functor. %In some cases (for example when $P$ is a CW poset) these cohomology theories coincide. 
The presheaf perspective will not be used in this paper, but more details can be found in the work of Turner and Everitt \cite{everitt2009homology, turnereverittcellular}. 

Let $\mc{A}$ be an abelian category, $P$ a finite thin poset with balanced coloring $c$, and $F:P\to\mc{A}$ a covariant functor. Define a cochain complex, denoted $C^*(F,c)$ with differential $\delta$ defined by the formulas
\begin{equation}\label{pococomplex}
    C^k(F,c)=\bigoplus_{\rk(x)=k}F(x)\ \ \ \ \ \ \ \ \ \delta^k=\mathop{\sum_{x\lessdot y }}_{\rk(x)=k} c(x\lessdot y)F(x\lessdot y).
\end{equation}
Lemma \ref{welldef} shows $\delta^2=0$ so these complexes are well defined. Denote the cohomology of $C^*(F,c)$ by $H^*(F,c)$. For a more complete treatment of the association of homology theories to functors on thin posets, see \cite{chandlerthin}. In the above definition, morphisms from $C^k$ to $C^{k+1}$ can be identified as matrices with columns indexed by elements in $P$ of rank $k$ and rows indexes by elements in $P$ of rank $k+1$, where the matrix element indexed by elements $x$ and $y$ (with rank $i$ and $i+1$ resp.) is a morphism $F(x)\to F(y)$. Then $\delta^k$ can be identified as the matrix with $c(x\lessdot y)F(x\lessdot y)$ in the matrix element indexed by $x$ and $y$.

\begin{remark}
In \cite{chandlerthin}, the complex which we write as $C^*(F,c)$ is denoted $C^*(P,\mc{A},F,c)$, but for our purposes in this paper, we choose the shorter notation. 
\end{remark}

\begin{lemma} $C^*(F,c)$ as defined in Equation \ref{pococomplex} is a cochain complex.  \label{welldef} 
\end{lemma}
\begin{proof} Consider the matrix of $d^2$ with respect to the direct sum decomposition $C^k=\oplus_{\rk(x)=k}F(x)$, $C^{k+2}=\oplus_{\rk(y)=k+2}F(y)$. Let $\rk(x_0)=k$ and $\rk(y_0)=k+2$. The closed interval from $x_0$ to $y_0$ in $P$ is a diamond $\{x_0,a,b,y_0\}$. Thus the $(x_0,y_0)$-component of the matrix of $d^2$ is $c(a\lessdot y_0)F(a\lessdot y_0)\circ c(x_0\lessdot a)F(x_0\lessdot a)+c(b\lessdot y_0)F(b\lessdot y_0)\circ c(x_0\lessdot b)F(x_0\lessdot b)$. Since $F$ is commutative on diamonds, this is equal to $[c(a\lessdot y_0)c(x_0\lessdot a)+c(b\lessdot y_0)c(x_0\lessdot b)]\cdot [F(a\lessdot y_0)\circ F(x_0\lessdot a)]=0$ since $c$ is a balanced coloring. 
\end{proof}

Examples of thin posets include face posets of polytopes, Bruhat orders on Coxeter groups, face posets of regular CW complexes (CW posets), and Eulerian posets \cite{bjorner1984posets, stanley1994survey}.
In 1984, Bj\"{o}rner proved that CW posets are exactly those posets with $\hat{0}$ whose lower intervals are homeomorphic to spheres (that is, the order complex of each open interval  $(\hat{0},x)$ is homeomorphic to a sphere) \cite{bjorner1984posets}. In particular, any thin shellable poset is a CW poset \cite[Proposition 4.5]{bjorner1984posets}. 
The following theorem indicates that in the case of CW posets, the association of cohomology theories $(F:P\to\mc{A})\mapsto H^*(F)$ from Equation \ref{pococomplex} is particularly well behaved. 
\begin{theorem}
    \label{functconst}
    \label{cwbalanced}
    \label{balind}
    Let $P$ be a CW poset. Then
    \begin{enumerate}
        \item (\cite[Corollary 2.7.14]{bjorner2006combinatorics}) $P$ is balanced colorable.
        \item  ({\cite[Theorem 6.1]{chandlerthin}}) Any labeling of vertices and edges of the Hasse diagram of $P$ by objects and morphisms of some category $\mc{A}$ which commutes on diamonds of $P$ determines a functor $F:P\to \mc{A}$.
        \item ({\cite[Theorem 6.6]{chandlerthin}}) Let $F:P\to\mc{A}$ be a functor to an abelian category. Then the complex $C^*(F,c)$ is independent of $c$ up to natural isomorphism.
    \end{enumerate}
\end{theorem}

\begin{remark}
According to Theorem \ref{balind}, we may omit $c$ from the expression $H^*(F,c)$ to simplify our notation to $H^*(F)$.
\end{remark}

Let $\mc{A}$ be an abelian category, and let $\mc{C}^b(\mc{A})$ denote the category of bounded cochain complexes $C$ in $\mc{A}$, that is, such that $C^i=0$ for all but finitely many $i$. Note that $\mc{C}^b(\mc{A})$ is also an abelian category. Let $K_0(\mc{A})$ denote the \textit{Grothendieck group} of $\mc{A}$, that is, the free abelian group generated by the set of isomorphism classes $\{[X] \ | \ X\in\Ob(\mc{C})\}$ modulo the relations $[X]=[X^\prime]+[X^{\prime\prime}]$ for each short exact sequence $0\to X^\prime\to X\to X^{\prime\prime}\to 0$ in $\mc{A}$. Additionally, if $\mc{A}$ has a monoidal structure $\otimes$, then $K_0(\mc{A})$ inherits the structure of a ring, with product $[X]\cdot[Y]=[X\otimes Y]$.

\begin{theorem}[{\cite[Theorem 9.2.2]{weibel2013k}}]
\label{eulercharacteristic}
Let $\mc{A}$ be a monoidal abelian category. The Euler characteristic  $\chi:K_0(\mc{C}^b(\mc{A}))\to K_0(\mc{A})$ 
 that sends $  [C]\mapsto \sum_{i\in\Z} (-1)^i[H^i(C)]$ is a ring isomorphism. 
 Furthermore, we have
\begin{equation}
\sum_{i\in\Z} (-1)^i[H^i(C)]=\sum_{i\in\Z} (-1)^i[C^i]
\end{equation}
\end{theorem}
From now on, we identify $K_0(\mc{C}^b(\mc{A}))$ with $K_0(\mc{A})$ via the isomorphism in Theorem \ref{eulercharacteristic}.

\begin{ex}\label{euler}
Let $R$ be a principal ideal domain. Let $\Rmod$ denote the (abelian) category of finitely generated (left) $R$-modules. There is an isomorphism of rings:
$$K_0(\Rmod)\cong \Z \ \ \ \ \ \ \ \ \ \ \ \  [M]\mapsto \rk M$$
(see for example \cite[Section 2]{weibel2013k}). In light of Theorem \ref{eulercharacteristic} we have the following identification of rings:
\[K_0(\mc{C}^b(\Rmod))\cong \Z \ \ \ \ \ \ \ \ \ \ \ \  [C]\mapsto \sum_{i\in\Z}(-1)^i\rk C^i.\]
\end{ex}

\begin{ex}\label{gradedeuler}
Let $R$ be a principal ideal domain. Let $\Rgmod$ denote the (abelian) category of finitely generated graded (left) $R$-modules. Given a graded module $M=\oplus_{i\in\Z}M^i\in\Rgmod$, the \textit{graded rank} of $M$ is $\qrk M=\sum_{i\in\Z}q^i\rk M^i$. Let $\Z[q,q^{-1}]$ denote the ring of Laurent polynomials in the variable $q$. There is an isomorphism of rings:
$$K_0(\Rgmod)\cong \Z[q,q^{-1}] \ \ \ \ \ \ \ \ \ \ \ \  [M]\mapsto \qrk M=\sum_{i\in\Z} q^i\rk M^i$$
(see \cite[Example 7.14]{weibel2013k}).
In light of Theorem \ref{eulercharacteristic} we have the following identification of rings:
$$K_0(\mc{C}^b(\Rgmod))\cong \Z[q,q^{-1}] \ \ \ \ \ \ \ \ \ \ \ \  [C]\mapsto \sum_{i\in\Z}(-1)^i\qrk C^i=\sum_{i,j\in\Z}(-1)^i q^j \rk C^{i,j}.$$
The map $[C]\mapsto \sum_{i,j\in\Z}(-1)^i q^j \rk C^{i,j}$ is usually referred to as the \textit{graded Euler characteristic} (see for example \cite{khovanov1999categorification} or \cite{bar2002khovanov}).
\end{ex}

\begin{theorem}\label{pococateg}
Let $\mc{A}$ be a monoidal abelian category, $P$ a balanced colorable thin poset, and $F:P\to\mc{A}$ a functor. Then in $K_0(\mc{C}^b(\mc{A}))$ we have the following equality:
\[[H^*(P,F)]=\sum_{x\in P}(-1)^{\rk(x)}[F(x)].\]
\end{theorem}
\begin{proof}
We identify $H^*(P,F)$ in $\mc{C}^b(\mc{A})$ as a chain complex with all differentials being zero maps. Then by Theorem \ref{eulercharacteristic}, we have $[H^*(P,F)]=[C(P,F,c)]$ where $c$ is a balanced coloring of $P$. Since $C^i(P,F,c)=\oplus_{\rk(x)=i}F(x)$, the desired equality follows from Theorem \ref{eulercharacteristic}.
\end{proof}
Theorem \ref{pococateg} provides the following categorification technique. To categorify an element $g$ of a ring $R$, one should find a monoidal abelian category $\mc{A}$ with an isomorphism $\phi:K_0(\mc{A})\to R$, and an object $G\in\mc{C}^b(\mc{A})$ such that $[G]=g$ under the identification $\phi\circ\chi$ of $K_0(\mc{C}^b(\mc{A}))$ with $R$. In this case we say that $G$ categorifies $g$. Given an object (for example, a poset invariant) of the form $g(P)=\sum_{x\in P} (-1)^{\rk(x)}g(x)$ where $g:P\to R$ is a function from a thin poset $P$ to a ring $R$ (such an expression is called a \text{rank alternator} for $f$ in \cite{chandlerthin}), one can categorify $g$ by constructing a functor $G:P\to\mc{A}$, where $K_0(\mc{A})\cong R$ and $[G(x)]=g(x)$, in which case we have $[H^*(G)]=g(P)$.

%Many well known cohomology theories can be constructed in this way, the most famous of which being Khovanov homology \cite{bar2002khovanov}. In this paper, we use algebraic Morse theory to give a cancellation technique for computing homology of such chain complexes, and give an application to the chromatic homology \cite{helme2005categorification} of Helme-Guizon and Rong. 

\section{Algebraic Morse Theory and an Application to Thin Poset Cohomology}\label{sectionmorse}

Morse theory for simplicial complexes, defined by Forman \cite{forman2002user}, provides tools that greatly simplify homology calculations. 
In \cite{skoldberg2006morse}, Skoldberg describes a version of discrete Morse theory for arbitrary chain complexes (also done independently by J\"{o}llenbeck and Welker in \cite{jollenbeck2005resolution}). In this section we review some results from \cite{skoldberg2006morse} and give an application of the main theorem of algebraic Morse theory in the setting of thin posets and their associated cohomology theories.

\begin{defn}
A \textit{based complex} of $R$-modules is a cochain complex $(K,d)$ of $R$-modules with a direct sum decomposition $K_n=\oplus_{\alpha\in I_n}K_\alpha$ where the $I_n$ are mutually disjoint index sets. Denote the cohomology of a based complex $K$ by $H^*(K)$. For $\alpha\in I_m$ and $\beta\in I_{m+1}$, let $d_{\beta,\alpha}$ denote the component of $d$ from $K_\alpha$ to $K_\beta$:
$$d_{\beta,\alpha}=K_\alpha\hookrightarrow K_m\xrightarrow[]{d}K_{m+1}\twoheadrightarrow K_\beta.$$ Given a based complex $K$, one can construct a directed graph $G_K$ with vertex set $V=\cup_nI_n$ and a directed edge $\alpha\to\beta$ whenever $d_{\beta,\alpha}\neq 0$. 
\end{defn}

 A \textit{matching} $M$ on a directed graph $D=(V,E)$ is a collection of disjoint edges in $E$. A matching $M$ is called \textit{complete} if every vertex in $D$ is incident with some edge in $M$. 
Define $D^M$ to be the directed graph  gotten from $D$ by reversing the direction of all arrows in $M$. For a based complex $K$, say a matching $M$ on $G_K$ is \textit{acyclic} if there are no directed cycles in $G_K^M$. Call the matching $M$ \textit{Morse} if $M$ is acyclic and for each edge $\alpha\to\beta$ in $M$, $d_{\beta,\alpha}$ is an isomorphism. A vertex $v$ in $G_K^M$ is called \textit{$M$-critical} if $v$ is not incident with any of the edges in $M$. Let $M^0$ denote the set of $M$-critical vertices in $G_K$, and let $M^1$ denote the set of vertices in $G_K$ which are incident with some edge in $M$. 
The following theorem is essential for our results.

\begin{theorem}\label{morsetheorem}\label{acycliccorollary}\label{subcomplexacyclic}
Let $K$ be a based complex, and let $M$ be a Morse matching on $K$.
\begin{enumerate}
    \item If $L=\oplus_{\alpha\in M^0}K_\alpha$ is a subcomplex of $K$, then $K$ is homotopy equivalent to $L$.
    \item If $M$ is complete, then $K$ is homotopy equivalent to the zero complex.
    \item If $L=\oplus_{\alpha\in M^1}K_\alpha$ is a subcomplex, then $H^*(K)\cong H^*(K/L)$.
\end{enumerate}
\end{theorem}

\begin{proof}
Part 1 is stated in \cite{skoldberg2006morse} Corollary 2. For part 2, since there are no critical cells, this is immediate from part 1. For part 3, we begin with the short exact sequence $0\to L\to K\to K/L\to 0$ and obtain a long exact sequence
\[\dots\to H^r(L)\to H^r(K)\to H^r(K/L)\to H^{r+1}(L)\to\dots\]
By Corollary \ref{acycliccorollary}, $H^r(L)=0$ for all $r$ so by exactness we have $H^r(K)\cong H^r(K/L)$ for all $r$. 
\end{proof}

Next we use this result in the setting of cohomology of functors on thin posets. First however notice that the above discussion was stated for categories of modules (following \cite{skoldberg2006morse}), but everything works just as well over arbitrary abelian categories. Let $P$ be a thin poset with balanced coloring $c$ and a functor $F:P\to\mc{A}$. Notice by definition that $C^*(F,c)$ is a based complex $K$ with index set $I_n=\{x\in P \ | \ \rk(x)=n\}$, $K_x=F(x)$ and for $x,y\in P$, \begin{equation}\label{thintobased}
    d_{y,x}=\begin{cases}c(x\lessdot y)F(x\lessdot y)& \text{if}\ x\lessdot y\\ 0&\text{otherwise}.\end{cases}
\end{equation} 
Notice that any upper or lower order ideal $\mc{I}$ of $P$, is also a thin poset, and thus we can restrict $F$ and $c$ to $\mc{I}$ obtaining the complex $C^*(F|_\mc{I}, c|_\mc{I})$.

\begin{theorem}\label{morsethinposet}
Let $P$ be a thin poset with balanced coloring $c$ and $F:P\to\mc{A}$ a functor, where $\mc{A}$ is an abelian category. Suppose that $\mc{I}$ is an upper or lower order ideal in $P$ and there is a complete acyclic matching $M$ on $\mc{I}$ such that $F(e)$ is an isomorphism for each $e\in M$. Then 
\[H^*(F)\cong H^*(F|_{P\setminus\mc{I}}).\]
\end{theorem}
\begin{proof}
Notice that in the based complex corresponding to $C^*(P,F)$ as defined in \ref{thintobased}, $M$ becomes a Morse matching. In the case that $\mc{I}$ is an upper order ideal, it follows that $C^*(F|_\mc{I}, c|_\mc{I})$ is a subcomplex of $C^*(F,c)$ and Corollary \ref{subcomplexacyclic} gives the desired result since the quotient of $C^*(F,c)$ by the subcomplex $C^*(F|_\mc{I}, c|_\mc{I})$ is isomorphic to $C^*(F|_{P\setminus\mc{I}}, c|_{P\setminus\mc{I}})$. In the case that $\mc{I}$ is a lower order ideal, then $M^0=P\setminus\mc{I}$. Since $P\setminus \mc{I}$ is an upper order ideal, it follows that $\bigoplus_{x\in M^0} F(x)$ is the subcomplex $C^*(F|_{P\setminus\mc{I}}, c|_{P\setminus\mc{I}})$ and thus Theorem \ref{morsetheorem} tells us that $C^*(F,c)$ is homotopy equivalent to $C^*(F|_{P\setminus\mc{I}}, c|_{P\setminus\mc{I}})$, and so in particular, $H^*(F)\cong H^*(F|_{P\setminus\mc{I}})$.
\end{proof}

\section{An Acyclic Matching on Spanning Subgraphs}
\label{whitney}

Let $G=(V,E)$ be a finite graph. That is, $V$ is a finite set and $E\subseteq \binom{V}{2}$ is a collection of subsets of size two. A \textit{cycle} (of length $k$) in $G$ is a sequence $x_1,x_2,\dots,x_k,x_1$ of $k$ distinct vertices in $G$ such that $\{x_i,x_{i+1}\}$ is an edge of $G$ for $1\leq i\leq k$ where subscripts are taken modulo $k$. The \textit{edge space} of $G$ is the $\Z_2$-vector space $\mathcal{E}(G)$ with basis $E$, and the \textit{cycle space} of $G$ is the subspace $\mathcal{C}(G)$ of $\mathcal{E}(G)$ spanned by cycles. We will identify each subset $S\subseteq E$ with the element $S=\sum_{e\in S}e\in \mathcal{E}(G)$, and also with the spanning subgraph of $G$ with vertex set $V$ and edge set $S$. Given $S,T\in \mathcal{E}(G)$, $S+T$ is thus identified with the symmetric difference $S\Delta T=S\cup T-S\cap T$. The  well-known fact that a connected graph has an Eulerian circuit if and only if every vertex has even degree (this goes all the way back to Euler's solution to the bridges of K\"{o}nigsberg problem) takes the following form in the present context.

\begin{lemma}
Given $S\subseteq E$, we have $S\in \mathcal{C}(G)$ if and only if $S$ is a union of edge-disjoint cycles.
\label{symmetricdiff}
\end{lemma}
\begin{proof}
The result in \cite[Proposition 1.9.2]{diestel1997graph} gives that $S\in\mathcal{C}(G)$ if and only if each vertex in $S$ has even degree. Thus we must show that $S$ is a union of edge-disjoint cycles if an only if each vertex in $S$ has even degree. The forward direction is immediate since given a vertex $v$, each cycle containing $v$ contributes 2 to the degree of $v$. We show the backwards direction by induction on the number of edges. If $|S|=0$ and each vertex has even degree, then $S$ is the empty union of cycles. For $|S|>0$, there must exist a vertex $v$ of degree $\geq 2$. Consider a walk in $S$ starting at $v$. Pick an edge $e_0\in S$ containing $v=v_0$ and walk along $e_0$ to reach the vertex $v_1$. Since $v_1$ has positive even degree, there must be an edge $e_1\neq e_0$ containing $v_1$. Walk along this edge to arrive at a vertex $v_2$. Since $S$ is finite, this process must eventually reach a vertex $v_k$ for which $v_{k+1}=v_i$ for some $i<k-1$. Thus $S$ contains a cycle $C=e_i+e_{i+1}+\dots+e_k\in\mathcal{C}(G)$. Consider the subgraph $S^\prime =S+C$. Consider the degree of a vertex $v$ in $S^\prime$. If $v$ is incident with an edge in $C$, then the degree of $v$ in $S^\prime$ is two less than it is in $S$, and thus the degree of $v$ in $S^\prime$ is even. If $v$ is not incident with an edge in $C$, then the degree of $v$ in $S^\prime$ is the same as the degree of $v$ in $S$. We conclude that every vertex has even degree in $S^\prime$. By induction, $S^\prime=C_1+\dots +C_k$ where each $C_i$ is a cycle and $C_i\cap C_j=\varnothing$ for $i\neq j$. Therefore $S=C+C_1+\dots+C_k$ is an edge disjoint union of cycles.
\end{proof}

\begin{defn}[\cite{whitney1932logical}]
\label{brokencircdef}
Let $G=(V,E)$ be a graph, and let $\mc{O}$ be a fixed ordering $\mc{O}=(e_1,\dots, e_m)$ of the edge set $E$. A \textit{broken circuit} in $G$ is $C+e\in\mathcal{E}(G)$ (or $C\setminus\{e\}$ as sets) where $C\subseteq E$ is a cycle in $G$ and $e$ is the edge in $C$ with the largest label. Define $\NBC_{G,\mc{O}}$ (or just $\NBC$ when $G$ and $\mc{O}$ are clear from context) to be the set of all $S\subset E$ such that $S$ contains no broken circuits, and 
 let $\BCC_{G,\mc{O}}$ (or just $\BCC$) be its complement, that is,  $\BCC=2^E\setminus \NBC$. 
\end{defn}

Recall from Section \ref{sectionmorse} that $M^0$ denotes the set of $M$-critical vertices of a matching $M$ on a directed graph $D$. Given a graph $G=(V,E)$ and $S\subseteq E$, let $k(S)$ denote the number of connected components in the spanning subgraph $S$ (the graph with vertex set $V$ and edge set $S$), and let $\lambda(S)$ denote the set partition of $V$ into the connected components of $S$.
 
\begin{lemma} Let $G=(V,E)$ be a graph with a fixed ordering of its edges. The Hasse diagram of the Boolean lattice $2^E$ has an acyclic matching $\MBC$ in which each matched pair $S\lessdot S^\prime$ satisfies $k(S)=k(S^\prime)$, $\lambda(S)=\lambda(S^\prime)$, and $\MBC^0=\NBC$.
\label{bcmatching}
\end{lemma}
\begin{proof}
Given a function $f:S\to S$ on a set $S$, recall the \textit{orbit} of $x\in S$ is the set $\{f^n(x) \ | \ n\geq 0\}$. We construct a matching on the Hasse diagram of $2^E$ as follows. First we construct an involution $i:\BCC\to\BCC$ (that is, a function $i:\BCC\to\BCC$ such that $i^2$ is the identity function). Since $i$ is an involution, the orbits of $i$ form a collection of disjoint 2-element sets which covers $\BCC$. We define $i$ in such a way that for any $S\in\BCC$, $i(S)$ either covers $S$ or is covered by $S$ in $\BCC$. Therefore the orbits of $i$ form a matching on the Hasse diagram of $2^E$, which we denote $M_\BCC$, for which $M_\BCC^0=2^E\setminus \BCC=\NBC$. We now define the involution $i$.

Given $S\in\BCC$, pick $e\in E$ maximally among all broken circuits of the form $C+e$ in $S$, where $e\in C$. Define $i(S)=S+e$ (recall $S+e$ is either $S\cup\{e\}$ or $S\setminus\{e\}$ depending whether $e\in S$).
Let $C^\prime+f\subseteq i(S)$ be a broken circuit with $f$ chosen maximally among all such broken circuits. Suppose that $e\neq f$. Then $f>e$ since $C+e$ is a broken circuit in $i(S)$, and thus $f\notin C$ (else we contradict maximality of $e$). Similarly,we have $e\in C^\prime$ since otherwise $C^\prime+f$ is a broken circuit in $S$ again contradicting maximality of $e$. Lemma \ref{symmetricdiff} tells us that $C+C^\prime$ is a union of cycles, and we know that $e\notin C+C^\prime$ and $f\in C+C^\prime$ so $C+C^\prime+f\subseteq S$ contains a broken circuit missing the edge $f$, thus contradicting maximality of $e$ among broken circuits in $S$.

Let $\MBC$ denote the matching consisting of the collection of orbits of $i$. To see that $\MBC$ is acyclic, it suffices to give a linear extension of $\BCC$ in which $T$ follows $S$ for each pair $(S\lessdot T)$ in $\MBC$ (this is  \cite[Theorem 11.2]{kozlov2007combinatorial}). For each $e=(S\lessdot T)\in M$ let $u(e)=T$ and $d(e)=S$ ($u$ for `up' and $d$ for `down'). Let $u(M)=\{u(e) \ | \ e\in M\}$ and $d(M)=\{d(e) \ | \ e\in M\}$. Fix a total ordering on $d(M)$ by setting $U\leq V$ whenever $|U|<|V|$ or $|U|=|V|$ and $U$ is lexicographically larger than $V$ (that is, order ranks by reverse lexicographic order). Let $\MBC=\{(S_1,T_1),\dots (S_n,T_n)\}$ with the $S_i$ ordered as indicated above and $T_i=S_i+ e_i$ where $S_i$ has a distinguished broken circuit $C_i+e_i$. We will now show that $S_1,T_1,S_2,T_2,\dots,S_n,T_n$ is such a linear extension of the inclusion ordering on $\BCC$ (see for example Figure \ref{example}).
\begin{figure}[h]
    \centering
    \begin{tikzpicture}
        %\draw[help lines] (-3,0) grid (10,8);
        \node[scale=.6] at (6.8,5){\usebox\bcexample};
        \node at (1.8,5.6){$\BCC=$};
        \node at (-2.5,2.55){$G=$};
        \node at (-.5,2.5){\usebox\graph};
        \node at (4,.5){\usebox\linearextension};
    \end{tikzpicture}
    \caption{An example of the acyclic matching on $\BCC$ for the graph $G$ shown above, with matched edges shown in red. Ranks are ordered reverse lexicographically from top to bottom. The resulting linear extension is shown below.}
    \label{example}
\end{figure}
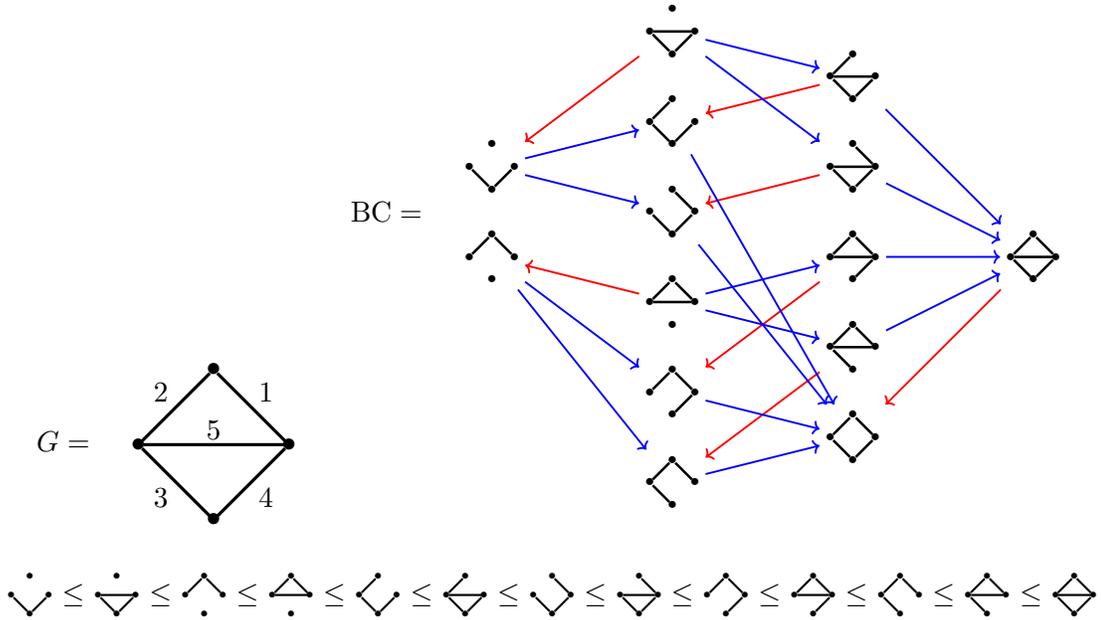
Suppose that $U,V\in\BCC$ and $U\subset V$. We now show by cases that $U$ comes before $V$ in the linear ordering $S_1,T_1,S_2,T_2,\dots,S_n,T_n$.
\setcounter{case}{0}
\begin{case}
$U,V\in d(M)$. Then $U$ comes before $V$ in the linear ordering because $|U|<|V|$. 
\end{case}

\begin{case} $U\in u(M), V\in d(M)$. Then $U=T_i$ and $V=S_j$ for some $i,j\in[n]$. Therefore $S_i\subsetneq T_i\subseteq S_j$ and thus $|S_i|<|S_j|$ so $S_i$ comes before $S_j$ in the linear extension by definition. Since $T_i$ comes directly after $S_i$, we conclude $U=T_i$ comes before $V=S_j$.
\end{case}

\begin{case} $U, V\in u(M)$. Let $U=T_i$ and $V=T_j$ for some $i,j\in[n]$. Then $|S_i|=|T_i|-1$ and $|S_j|=|T_j|-1$ and so $|S_i|<|S_j|$. We conclude that $S_i$ comes before $S_j$ in the linear extension and therefore $T_i$ comes before $T_j$. 
\end{case}

\begin{case}
$U\in d(M), V\in u(M)$. Let $U=S_i$ and $V=T_j$ for some $i,j\in[n]$.
If $S_i\subseteq S_j$ then $S_i$ comes weakly before $S_j$ which comes before $T_j$ in the linear extension, so we are done.
Similarly, if $|S_i|<|S_j|$ we are done so let us assume $|S_i|=|S_j|$. Thus there exists $x\in S_j$ with $x\notin S_i$, $T_j=S_i+x=S_j+e_j$. Notice also that $x,e_i,$ and $e_j$ are mutually distinct edges. See Figure \ref{schematic} for a schematic diagram of $S_i,S_j,$ and $T_j$. It suffices to show that $S_i$ is lexicographically larger than $S_j$, or equivalently, that $e_j>x$ in our fixed edge ordering.

Suppose towards a contradiction that $e_j<x$. Then $x\notin C_j$ (else we contradict maximality of $e_j$ in $C_j$) and therefore $C_j\subseteq S_i$ (since $e_j\in S_i$) so we can conclude that $e_i>e_j$ by maximality of $e_i$ in $S_i$. Thus we find that $e_i\notin C_j$ and furthermore $e_j\in C_i$ since otherwise $C_i+e_i\subseteq S_j$ (which would contradict maximality of $e_j$ in $S_j$).
Lemma \ref{symmetricdiff} tells us that $C_i+C_j$ is an edge-disjoint union of cycles, and we have $e_i\in C_i+C_j$ and $e_j\notin C_i+C_j$ and therefore $C_i+C_j+e_i\subseteq S_j$ contains a broken circuit missing the edge $e_i$ and therefore $e_j>e_i$ (contradicting our earlier finding that $e_i>e_j$). Therefore we conclude that $e_j>x$ and therefore $S_i$ is lexicographically larger than $S_j$. \qedhere
%
%Notice that $C_i+C_j-e_i-e_j\subseteq S_J$ and $e_j\in C_i\cap C_j$ and thus $(C_i\cup C_j-C_i\cap C_j)-e_i\subseteq S_j$. But $(C_i\cup C_j-C_i\cap C_j)$ is a symmetric difference of cycles and so is a union of cycles, one of which contains $e_i$. Therefore $S_j$ contains a broken circuit of the form $C-e_i$, contradicting maximality of $e_j$ in $S_j$. Thus we conclude that $x<e_j$. 
%
%If $x\notin C_j$ then $C_j\subseteq S_i$ so $e_i>e_j$ and therefore $e_i\notin C_j$. We must have $e_j\in C_i$ since otherwise $C_i-e_i\subseteq S_j$ contradicting maximality of $e_j$ in $S_j$. But then $(C_i+C_j-e_j)-e_i$ is a broken circuit in $S_j$ contradicting maximality of $e_j$. Therefore $x\in C_j$ and we conclude that $e_j>x$ and therefore $S_i=S_j+e_j-x$ is lexicographically larger than $S_j$. 
\end{case}
\end{proof}

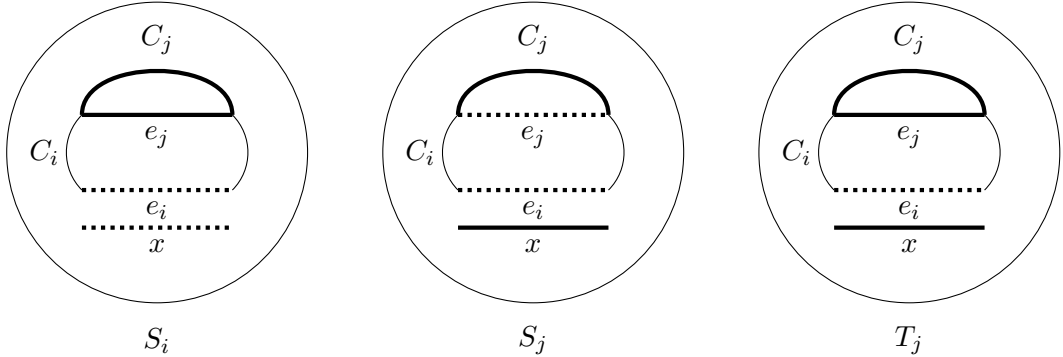
\begin{figure}
\centering
\begin{tikzpicture}
%\draw [help lines](0,0) grid (14,4);
\node at (2,-.5){$S_i$};
\node at (7,-.5){$S_j$};
\node at (12,-.5){$T_j$};
\node at (2,2){\usebox\si};
\node at (7,2){\usebox\sj};
\node at (12,2){\usebox\tj};
\end{tikzpicture}
\caption{A schematic for visualizing the subgraphs $S_i$, $S_j$, and $T_j$.}
\label{schematic}
\end{figure}

\section{A Broken Circuit Model for Chromatic Homology}
\label{brokenchromatic}

%In 1852, Francis Guthrie noticed that only four colors were needed to color the map of counties of England in such a way that no two adjacent counties get the same color, and went on to ask if this is a general phenomena of maps. This question became known as the ``four-color problem''. Due to the simplicity of its statement, and the fact that it remained unsolved for so long, this became one of the most famous open problems in mathematics. The four-color problem can be stated graph theoretically as follows.
A \textit{proper coloring} of a graph $G=(V,E)$ is a function $c:V\to[n]$ where $[n]=\{1,2,\dots,n\}$ such that if there is an edge from $x$ to $y$ in $G$ then $c(x)\neq c(y)$. A graph is \textit{$n$-colorable} if there exists a proper coloring $c:V\to[n]$. The famous four-color problem is equivalent to asking whether every planar graph is four-colorable. In 1912, Birkhoff \cite{birkhoff1912determinant} introduced the chromatic polynomial in an attempt to obtain an algebraic/analytic solution to the four-color problem. Although this attempt was unsuccessful, the idea contributed largely to the development of the field of algebraic graph theory. 

\begin{defn}[\cite{birkhoff1912determinant}]
Given a graph $G$, the \textit{chromatic polynomial} $\chi_G:\N\to\N$ is defined by $$\chi_G(x)=\#\{c:V\to[x]\ | \ c \ \text{is proper}\}.$$
\end{defn}

\begin{remark}
The four-color problem was answered in the affirmative in 1976 by  Appel and  Haken \cite{appel1977solution} essentially by using a computer to check that no possible counterexample can exist. A direct proof is still highly sought after. A possible new approach to the four-color problem is given by Kronheimer and Mrowka \cite{kronheimer2015tait} using instanton homology.
\end{remark}

It is not clear from the definition that $\chi_G$ is in fact a polynomial. The following result establishes this fact and gives a convenient way to calculate $\chi_G$. Recall, given a set $E$, $2^E$ denotes the collection of all subsets of $E$, partially ordered by containment.

\begin{lemma}[\cite{birkhoff1912determinant}]
For any graph $G=(V,E)$,
\[\chi_G(x)=\sum_{S\in 2^E}(-1)^{|S|}x^{k(S)}\]
where $k(S)$ is the number of connected components in $S$.
\end{lemma}

In 1932, Whitney found the following connection between the chromatic polynomial of $G$ and broken circuits in $G$, resulting in a significant simplification in the computation of chromatic polynomials. 
 
\begin{theorem}[\cite{whitney1932logical}]\label{whitneytheorem}
For any graph $G$, and any fixed ordering of its edge set,
 $$\sum_{S\in\BCC}(-1)^{|S|}x^{k(S)}=0$$
and therefore
 \[\chi_G(x)=\sum_{S\in\NBC}(-1)^{|S|}x^{k(S)}.\]
\end{theorem}

\begin{proof}
This follows directly from Lemma \ref{bcmatching} since every contribution of the form $(-1)^{|S|}x^{k(S)}$ to $\chi_G(x)$ where $S\in\BCC$ is matched with another contribution $(-1)^{|S^\prime|}x^{k(S^\prime)}$ where $S^\prime\in\BCC$, $|S^\prime|=|S|\pm 1$ and $k(S)=k(S^\prime)$. Thus all  contributions from $\BCC$ cancel and we are left with only the contributions from $\NBC$. 
\end{proof}

\begin{comment}
\begin{remark}
The matching constructed in the proof of Lemma \ref{bcmatching} was known to Whitney. The new part of Lemma \ref{bcmatching} is that the matching is in fact acyclic.
\end{remark}
\end{comment}

Inspired by Khovanov link homology, Helme-Guizon and Rong \cite{helme2005categorification} constructed a homology theory which categorifies the chromatic polynomial $\chi_G(x)$. Next we describe their construction in the setting introduced in Section \ref{pocohom}. Let $\Rgmod$ denote the category whose objects are graded $R$-modules, and whose morphisms are degree-preserving  degree preserving homomorphisms of graded $R$-modules. In this section, $A$ will always denote a commutative graded $R$-algebra $A=\oplus_{k\in\Z}A_i$ with identity such that each $A_i$ is free of finite rank.

\begin{defn}\label{peredge}
Given a graph $G=(V,E)$, a commutative ring $R$, and an $R$-algebra $A$ with multiplication $m:A\otimes A\to A$, the $(G,A)$-\textit{chromatic functor}, $\FCH_{G,A}:2^E\to\Rgmod$ (or just $\FCH$ when $G$ and $A$ are understood) is defined as follows:

\begin{enumerate}
    \item $\FCH(S)=A^{\otimes k(S)}$ for $S\in 2^E$, where $k(S)$ denotes the number of connected components in $S$. 
    \item Given a cover relation $S\lessdot S+e$ in $2^E$, adding the edge $e$ to $S$ either joins two distinct connected components, or completes a cycle. In the former case, $\FCH(S\lessdot S+e)$ is defined by multiplying tensor factors of $A$ corresponding to the two joining connected components. In the latter case, $\FCH(S\lessdot S+e)$ is the identity map. 
\end{enumerate}
\end{defn}

\begin{lemma}
For any graph $G$, and any $R$-algebra $A$, Definition \ref{peredge}  determines a functor \[\FCH_{G,A}:2^E\to\Rgmod.\]
\end{lemma}
\begin{proof}
Since the Boolean lattice is a CW poset (the face poset of the simplex), by Theorem \ref{functconst} it suffices to check that the ``per edge maps'' defined in Definition \ref{peredge} part 2 commute on diamonds. Given $S\in 2^E$, consider $e,f\in E\setminus S$ and the corresponding diamond $\{S,S+e,S+f,S+e+f\}$. 
\setcounter{case}{0}

\begin{case} If $k(S+e+f)=k(S)$, then all maps are identity maps, and the result is trivial.
\end{case}

\begin{case}If $k(S+e+f)=k(S)+1$, then both $\FCH(S+e\lessdot S+e+f)\circ\FCH(S\lessdot S+e)$ and $\FCH(S+f\lessdot S+e+f)\circ\FCH(S\lessdot S+f)$ consist of multiplying the two tensor factors corresponding to the two components of $S$ which are joined in $S+e+f$. 
\end{case}

\begin{case} If $k(S+e+f)=k(S)+2$, then both $\FCH(S+e\lessdot S+e+f)\circ\FCH(S\lessdot S+e)$ and $\FCH(S+f\lessdot S+e+f)\circ\FCH(S\lessdot S+f)$ consist of multiplying the three tensor factors corresponding to the three components of $S$ which are joined in $S+e+f$. These maps are equal by associativity in $A$. \qedhere
\end{case} 
\end{proof}

Given a graph $G=(V,E)$ where $E$ has a fixed ordering $E=\{e_1,\dots,e_m\}$, recall  the Boolean lattice $2^E$ has a balanced coloring $c:C(2^E)\to\{1,-1\}$ defined by $c(S\lessdot S+e_i)=(-1)^{\#\{j\in S \ | \ j<i\}}$. 
As per the construction in Section \ref{pocohom}, this data determines a chain complex $C^*(\FCH_{G,A},c)$ with homology $H^*(\FCH_{G,A})$ 
which we denote by $\HCH_A(G)$. The following theorem of Helme-Guizon and Rong shows that the homology $\HCH_A(G)$ categorifies the chromatic polynomial $\chi_G(x)$.

\begin{theorem}[\cite{helme2005categorification}]
Let $A$ be a commutative graded $R$-algebra with identity such that each $A_i$ is free of finite rank. For any graph $G$, under the identification 
$K_0(\mc{C}^b(\Rgmod))\cong \Z[q,q^{-1}]$ sending the cochain complex to its Euler characteristics $ [C]\mapsto \sum_{i\in\Z}(-1)^i\qrk C^i,$
we have
\[[\HCH_A(G)]=\chi_G(\qrk A).\]
\end{theorem}
\begin{proof}
Follows from Theorem \ref{pococateg}.
\end{proof}

\begin{remark}
Helme-Guizon and Rong \cite{helme2005categorification}, in their original construction of chromatic homology, mainly worked over $A_2=\Z[x]/(x^2)$  so that $q\rk A_2=1+q$, so one can recover the chromatic polynomial from $[\HCH_{A_2}(G)]$ with the substitution $q=x-1$. 
It turns out that $\HCH_{A_2}(G)$ is determined by the  $\chi_G(x)$, hence interesting applications might arise from working with other algebras, see \cite{helme2005graph,pabiniak2009first, sazdanovic2018patterns}. 
In particular, Pabiniak, Przytycki and Sazdanovic \cite{pabiniak2009first} show that already $\HCH_{A_3}(G)$ over algebra $A_3$ is a stronger graph invariant than $\chi_G(x)$. 
\end{remark}
Note that $\NBC\subseteq 2^E$ is a lower order ideal in $2^E$ and is thus also a thin poset.
Hence one can consider the restriction of $\FCH$ to $\NBC$ and the restriction of $c$ to $\NBC$, yielding a chain complex $C^*(\FCH|_\NBC,c|_\NBC)$ with homology $H^*(\FCH|_\NBC)$.

\begin{theorem}[Categorification of Whitney's Broken Circuit theorem]
\label{categcwhitney}
Let $A=\oplus_{k\in\Z}A_i$ be a commutative, graded $R$-algebra with identity such that each $A_i$ is free of finite rank. For any graph $G=(V,E)$,
%\RS{ANY? It needs to be at least commutative? Plus it is not mentioned in the statement whuch is strange} 
and any fixed ordering of the edge set $E$,
\begin{equation}
\label{WBCCat}
H^*(\FCH)\cong H^*(\FCH|_\NBC).
\end{equation}
\end{theorem}
\begin{proof}
Recall that for a cover relation $S\lessdot S+e$ for which $e$ completes a cycle in $S$, the map $\FCH(S\lessdot S+e)$ is an isomorphism. The result follows from Theorem \ref{morsethinposet} and Lemma \ref{bcmatching}.
\end{proof}
\begin{remark}
As in Example \ref{gradedeuler}, taking the (graded) Euler characteristic of both sides  of the above isomorphism recovers the statement of Theorem \ref{whitneytheorem}. Therefore Theorem \ref{categcwhitney} can be viewed as a categorification of Whitney's Broken Circuit Theorem.
\end{remark}
Theorem \ref{WBCCat} allows one to replace the complex $C^*(\FCH,c)$ by $C^*(\FCH|_\NBC,c|_\NBC)$, thus giving homological bounds for chromatic homology as an immediate corollary. By construction of the chromatic chain complex, we have $C^*(\FCH_{G_1\amalg G_2})\cong C^*(\FCH_{G_1})\otimes C^*(\FCH_{G_2})$, and therefore we may restrict attention to connected graphs, without loss of generality.
\begin{corollary} 
For any connected graph $G$ with $n$ vertices, $\HCH_A(G)$ is supported in homological gradings $0\leq i\leq n-1$.
    \label{support}
\end{corollary}
\begin{proof} 
Proof follows from Theorem \ref{categcwhitney} along with the following fact. Any tree on $n$ vertices has $n-1$ edges, thus any spanning subgraph of $G$ with $n$ or more edges must contain a cycle, and consequently a broken circuit. 
\end{proof}

The bounds in Corollary \ref{support}  are not the best  available. The deletion-contraction long exact sequence in chromatic homology over any $A_m$ was used to show that $\HCH_{A_m}(G)$ is supported in homological gradings $(i,j)$ with $0\leq i\leq n-2$ \cite{helme2005torsion}. Furthermore, the number of vertices minus the number of blocks of a graph provides a lower bound for homological span of chromatic homology over any $A_m$ \cite[Theorem 44]{sazdanovic2018patterns}.
Theorem \ref{categcwhitney} does however give a more constructive explanation for these homological bounds. Furthermore, Theorem \ref{categcwhitney} explicitly indicates a much quicker way to implement the calculation of chromatic homology in general. That is, in the construction of the chromatic chain complex, remove all direct summands corresponding to spanning subgraphs which contain broken circuits, and all per-edge maps into and out of such summands. For complete graphs, this reduces from a computation time which grows quadratically in the number of vertices to one which grows linearly. Finally, the broken circuit model gives a more general approach for finding such homological bounds. In particular, it can also be used for chromatic symmetric homology (see Section \ref{brokensymmetric}), where there is no long exact sequence analogous to the one in the case of chromatic homology.

\section{A Broken Circuit Model for Chromatic Symmetric Homology}
\label{brokensymmetric}

 The {\textit chromatic symmetric function} $X_G(x_1,x_2,\ldots)$ was introduced in 1995 by Stanley \cite{stanley1995symmetric}. Chromatic symmetric function is a multi-variable generalization of the chromatic polynomial $\chi_G(x)$ in the sense that \[X_G(1,\ldots,1,0,\ldots) = \chi_G(k),\] where we have substituted 1 for the first $k$ coordinates $x_1,\dots x_k$, and 0 for the rest. The chromatic symmetric function for $G=(V,E)$ is defined as
\[X_G(\mathbf{x})=\sum_c x_{c(v_1)}\dots x_{c(v_n)},\]
where the sum is over all proper colorings $c:V(G)\to\N$ of $G$.
The chromatic symmetric function generalizes many properties of the chromatic polynomial, but in general, $X_G(\mathbf{x})$ is a stronger invariant than $\chi_G(x)$.  Recall the power sum symmetric functions $p_k(\mathbf{x})=\sum_{i=0}^\infty x_i^k$. Given any partition $\lambda=(\lambda_1,\dots,\lambda_n)$ we define $p_\lambda=p_{\lambda_1}\dots p_{\lambda_n}$. Any edge subset $T\subseteq E$ determines a partition $\lambda(T)$ of $n$, defined by the number of vertices in each connected component of $S$. The chromatic symmetric function, like the chromatic polynomial has a state sum formula:
\begin{equation}
    X_G=\sum_{T\in 2^E}(-1)^{|T|}p_{\lambda(T)}.
    \label{csdefn}
\end{equation}
Similar to the case for the chromatic polynomial, the chromatic symmetric function also has a significant simplification in terms of broken circuits.
\begin{theorem}[\cite{stanley1995symmetric}] 
\label{csbroken}
\label{csbrokenthm}
For any graph $G$, and any fixed ordering of the edge set, 
$$\sum_{T\in \BCC}(-1)^{|T|}p_{\lambda(T)}=0,$$
and therefore
    \[X_G(\mathbf{x})=\sum_{T\in \NBC}(-1)^{|T|}p_{\lambda(T)}.\]
\end{theorem}

\begin{proof}
This follows directly from Lemma \ref{bcmatching}.
\end{proof}
With equation \ref{csdefn} at hand, Sazdanovi\'{c} and Yip \cite{sazdanovic2018categorification} categorify $X_G$ in a way analogous to Helme-Guizon and Rong's categorification of $\chi_G(x)$. We now recall their construction, but state it in the language of Section \ref{pocohom}.
%Let $\{e_1,\dots, e_n\}$ denote the standard basis for $\R^n$ and let $\alpha_i=e_i-e_{i-1}$. 
Let $S^\lambda$ denote the irreducible $\C[S_n]$-modules indexed by $\lambda\vdash n$. 
%In particular, $S^{(n-1,1)}$ is the standard $n-1$ dimensional $\C[S_n]$-module with basis $\alpha_1,\dots,\alpha_{n-1}$. 
For $n\in\N$, let $\mc{L}_n$ denote the graded $\C[S_n]$-module
\begin{equation}
    \mathcal{L}_n=\wedge^*S^{(n-1,1)}=\bigoplus_{k=0}^{n-1}\wedge^kS^{(n-1,1)}=\bigoplus_{k=0}^{n-1}S^{(n-k,1^k)}\{k\}
\end{equation}
where $\{k\}$ denotes a grading shift up by $k$.
%For $n=1$, $\mc{L}_n=S^{(n)}$ is the trivial representation. 
Let $T\subseteq E(G)$ be a spanning subgraph of $G$ with $r$ connected components, whose vertex sets we denote by $B_1,\dots, B_r$. Let $b_i$ denote the number of vertices in $B_i$. The sets $B_1,\dots, B_r$ of vertices partition $V(G)$ so $S_{B_1}\times\dots\times S_{B_r}$ is a Young subgroup of $S_{V(G)}\cong S_n$. Define the graded $\C[S_n]$-module
\begin{equation}
    \mc{M}_T=\Ind_{S_{B_1}\times\dots\times S_{B_r}}^{S_{V(G)}}(\mc{L}_{b_1}\otimes\dots\otimes \mc{L}_{b_r})
\end{equation}
For each cover relation $T\lessdot T+e$ in the Boolean lattice $2^E$, Sazdanovi\'{c} and Yip define a map $d_{T\lessdot T+e}:\mc{M}_T\to\mc{M}_{T+e}$. If adding $e$ completes a cycle when added to $T$, $d_{T\lessdot T+e}$ is the identity map. In the case that $e$ does not complete a cycle, we refer the reader to \cite{sazdanovic2018categorification} for details. For the purposes of this section, the details of the per-edge maps in general are not needed. 

\begin{defn}
Given a graph $G=(V,E)$ with $|V|=n$, define the \textit{chromatic symmetric functor} $\FCS_G:2^E\to\CSNgmod$ by
\begin{enumerate}
    \item Given $T\subseteq E$, $\FCS(T)=\mathcal{M}_T$,
    \item Given a cover relation $T\lessdot T+e$, define $\FCS(T\lessdot T+e)=d_{T\lessdot T+e}$.
\end{enumerate}
\end{defn}
As per the construction in Section \ref{pocohom}, for any graph $G=(V,E)$, the data $(2^E,\FCS_G,c)$ defines a cochain complex $C^*(\FCS_G,c)$ in $\mc{C}^b(\CSNgmod)$, with homology $H^*(\FCS_G)$. We will refer to $H^*(\FCS_G)$ as the \textit{chromatic symmetric homology} of $G$, and denote it by $\HCS(G)$.
Let $R_n$ denote the Grothendieck group of the category $\CSNgmod$. 
That is, $R_n$ is the free abelian group on the isomorphism classes $[S^\lambda]$ of Specht modules indexed by all partitions $\lambda$ of $n$. 
Define the graded ring  $R=\oplus_{n\geq 0} R_n$ with multiplication defined as follows. 
Given $N\in \CSNgmod$ and $M\in\CSMgmod$ define  $[M]\cdot[N]=[\Ind_{S_n\times S_m}^{S_{n+m}} M\otimes N]$. 
One can also think of $R$ as the Grothendieck ring of the monoidal category of all symmetric group representations. Let $\Lambda_\C$ denote the ring of symmetric functions over $\C$. 

\begin{theorem}[{\cite[Section 7.3, Theorem 1]{fulton1997young}}]
The homomorphism $\ch:R\to\Lambda_\C$ of graded rings defined by sending the Specht modules to the Schur functions
$[S^\lambda]\mapsto s_\lambda$
is an isomorphism.
\label{symgroth}
\end{theorem}

From now on, we identify $R$ with $\Lambda_\C$ via the isomorphism $\ch$ from Theorem \ref{symgroth}. 

\begin{theorem}[\cite{sazdanovic2018categorification}]
For any graph $G$, $[\HCS(G)]=X_G(\mathbf{x})$
in $K_0(\mc{C}(\CSNgmod))$.
\end{theorem}

\begin{proof}
By the identification allowed by Theorem \ref{symgroth},  $p_n=\sum_{i=0}^{n-1}(-1)^i[S^{(n-i,1^i)}]$ and therefore, for each $T\subseteq E$, we have $[\mc{M}_T]=p_{\lambda(T)}$, and the rest of the proof follows immediately by construction, considering equation \ref{csdefn}.
\end{proof}

One can consider the restriction of $\FCS$ to $\NBC$ and the restriction of $c$ to $\NBC$, yielding a chain complex $C^*(\FCS|_\NBC,c|_\NBC)$ with homology $H^*(\FCS|_\NBC)$.

\begin{theorem}\label{categcswhitney}
For any graph $G=(V,E)$, and any fixed ordering of the edge set $E$,
\[H^*(\FCS)\cong H^*(\FCS|_\NBC).\]
\end{theorem}
\begin{proof}
Recall that for a cover relation $T\lessdot T+e$ for which $e$ completes a cycle in $T$, $\FCS(T\lessdot T+e)$ is an isomorphism. Therefore the result follows from Theorem \ref{morsethinposet} and Lemma \ref{bcmatching}.
\end{proof}
Taking the (graded) Euler characteristic of both sides (as per Example \ref{gradedeuler}) of the above isomorphism yields Theorem \ref{whitneytheorem}, and therefore Theorem \ref{categcswhitney} can be viewed as a categorification of Whitney's Broken Circuit Theorem for the chromatic symmetric function.

\begin{corollary}  For any connected graph $G$ with $n$ vertices, $\HCS(G)$ is supported in bigradings $(i,j)$ which satisfy $0\leq i\leq n-1$ and $0\leq j\leq i$.
\label{cssupport}
\end{corollary}

\begin{proof} For the bounds on $i$, the argument is the same as in the proof of Corollary \ref{support}. The bounds on $j$ follow immediately from the definition of $\FCS(T)$. 
\end{proof}
 
 \section{A Broken Circuit Model for Graph Configuration Spaces}\label{spectralsequencesec}
 
Eastwood and Huggett's approach to categorifying the chromatic polynomial \cite{eastwood2007euler} relies on ideas from topology such as configuration spaces and Leray sequence. Let $M$ be a manifold, and $G$ a graph with vertex set $\{v_1,\dots,v_n\}$. Let $M^G$ denote the corresponding graph configuration space, that is $M^{\times n}\setminus Z$ where $Z$ is the union of all closed submanifolds of the form $\{\vec{x}\in M^{\times n} \ | \ x_i=x_j\}$ where $\{v_i,v_j\}$ is an edge in $G$. In \cite{baranovsky2012graph}, Sazdanovi\'{c} and Baranovsky give a spectral sequence with $E_1$ page isomorphic to the chromatic chain complex $C^*(2^E,F^\Ch_{A,G},c)$ of Helme-Guizon and Rong, with $A=H^*(M;R)$, which converges to the relative cohomology $H^*(M^{\times n},Z;R)$. By Lefshetz duality and Poincare duality, the relative cohomology $H^*(M^{\times n},Z;R)$ is isomorphic to the cohomology of the graph configuration space $H^*(M^G;R)$. The spectral sequence used here is a special case of the following spectral sequence of Bendersky and Gitler \cite{bendersky1991cohomology}. Given a simplicial topological space $X$ and a collection $Z_\alpha$ of closed subspaces for $\alpha\in E$, let $Z$ be the union of all of the $Z_\alpha$, and for $s\subseteq E$, let $Z_s$ be the intersection of all $Z_\alpha$ for $\alpha\in E$. Then the relative cohomology $H^*(X,Z;R)$ can be computed as the total complex of the bicomplex $\mc{C}^{*,*}=\oplus_{s\subseteq E}C^*(Z_s;R)$ with \begin{align}C^{i,j}&=\bigoplus_{|s|=i}C^j(Z_s;R)
 \label{barsazspec1}\\
 d_h^{i,j}:C^{i,j}\to C^{i+1,j},\ \ &\ \ \ \phi\in C^j(Z_s;R)\mapsto \sum_{s\lessdot t}c(s\lessdot t)i^\#\phi \label{barsazspec2}
 \end{align}
 where $i:Z_{t}\to Z_s$ is the inclusion inducing the chain map $i^\#:C^*(Z_{s};R) \to C^*(Z_{t};R)$, and
  \[d_v^{i,j}:C^{i,j}\to C^{i,j+1},\ \ \ \ \ \phi\in C^j(Z_s;R)\mapsto \delta\phi \]
  where $\delta$ is the simplicial codifferential in the complex $C^*(Z_s;R)$.
  There are two spectral sequences associated to a bicomplex, $\mc{C}^{*,*}$, one associated to choosing the zero page $E_0$ to be $(\mc{C}^{*,*},d_h)$ the other, choosing the zero page to be $(\mc{C}^{*,*},d_v)$. By choosing $(\mc{C}^{*,*},d_v)$ as the zero page, Sazdanovi\'{c} and Baranovski obtain a spectral sequence with $E_1$ page $(E_1,d_1)$ where $E_1^{p,q}=\oplus_{|s|=p}H^q(Z_s;R)$ and $d_1^{p,q}=(d_h^{p,q})^*$, the induced horizontal differential. This spectral sequence converges to the total complex, which is isomorphic to the cohomology of the graph configuration space $H^*(M^G;R)\cong H^*(M^{\times n},Z;R)$. We make the following simplification of their spectral sequence. Given a graded ring $A$ we will let $A(q)$ denote the set of homogeneous elements of $A$ of degree $q$. 
  
  \begin{theorem}\label{brokenspectral}
  There is a spectral sequence converging to $H^*(M^G;R)$ with $E_1$ page given by the broken circuit model for the chromatic complex:
  \[E_1^{p,q}=\bigoplus_{s\in\NBC, |s|=p} [F_{G,A}^\Ch(s)](q)=\bigoplus_{s\in\NBC, |s|=p} [A^{\otimes k(s)}](q),\]
  where $A=H^*(M;R)$ and 
  \[d_1:E_1^{*,*}\to E_1^{*+1,*},\ \ \ \ \ d_1=\sum_{s,t\in\NBC, s\lessdot t} c(s\lessdot t)F^\Ch(s\lessdot t).\]
  \end{theorem}
  \begin{proof}
  Let $G$ be a graph with a fixed ordering $\mc{O}$ of its edge set, and let $\NBC=\NBC_{G,\mc{O}}$ and $\BCC=\BCC_{G,\mc{O}}$ as in Definition \ref{brokencircdef}. Consider the double complex $(\mc{C}^{*,*},d_h,d_v)$ from equations (\ref{barsazspec1}), (\ref{barsazspec2}). Form the double complex $(\mc{C}^{*,*}_\NBC,d_h^\NBC,d_v^\NBC)$ where 
  \[\mc{C}^{p,q}_\NBC=\bigoplus_{s\in\NBC, |s|=p}C^q(Z_s;R), \ \ \ \  \ \ \ \ \ \ \ \ \  d_h^\NBC=d_h|_{\mc{C}^{*,*}_\NBC},\ \ \ \ \ \ \ \ \ \  d_v^\NBC=d_v|_{\mc{C}^{*,*}_\NBC}.\]
  Consider also the double complex $(\mc{C}^{*,*}_\BCC,d_h^{\BCC},d_v^{\BCC})$, where
  \[\mc{C}^{p,q}_\BCC=\bigoplus_{s\in\BCC_G, |s|=p}C^q(Z_s;R),  \ \ \ \  \ \ \ \ \ \ \ \ \ d_h^{\BCC}=d_h|_{\mc{C}^{*,*}_\BCC},\ \ \ \ \ \ \ \ \ \  d_v^{\BCC}=d_v|_{\mc{C}^{*,*}_\BCC}.\]
  Since $\mc{C}^{*,*}_\BCC$ is a subcomplex of $\mc{C}^{*,*}$, we have a short exact sequence of double complexes
  \[0\to \mc{C}^{*,*}_\BCC\to \mc{C}^{*,*}\to \mc{C}^{*,*}_\NBC\to 0\]
  inducing a short exact sequence on total complexes
  \begin{equation}\label{totses}
      0\to \Tot\mc{C}^{*,*}_\BCC\to \Tot\mc{C}^{*,*}\to \Tot\mc{C}^{*,*}_\NBC\to 0.
  \end{equation}
  Consider the chain complex $(\mc{C}^{*,*}_\BCC,d_h^\BCC)$. Notice, by definition, we have \[(\mc{C}^{*,*}_\BCC,d_h^\BCC)=C^*(\BCC,F^\BCC,c)\] where $F^\BCC(s)=C^*(Z_s;R)$ and \[F^\BCC(s\lessdot t)=i^\#:C^*(Z_s;R)\to C^*(Z_t;R)\] where $i:Z_t\to Z_s$ is the inclusion map. In the case that $t=s\cup\{e\}$ where $e\notin s$ and $e$ completes a cycle in $s$, then $Z_s=Z_t$ and the inclusion $i:Z_t\to Z_s$ induces an isomorphism $C^*(Z_s;r)\cong C^*(Z_t;R)$. Therefore, by Lemma \ref{bcmatching}, $\BCC$ has a complete Morse matching with respect to the functor $F^\BCC$, and therefore $H^*(\mc{C}^{*,*}_\BCC,d_h^\BCC)=H^*(\BCC,F^\BCC,c)=0.$
  Since $\mc{C}^{*,*}_\BCC$ has exact rows, the Acyclic Assembly Lemma \cite[Lemma 2.7.3]{weibel1995introduction} tells us that $\Tot \mc{C}^{*,*}_\BCC$ is acyclic, and therefore the long exact sequence on homology resulting from the short exact sequence (\ref{totses}) tells us $H(\Tot \mc{C}^{*,*}_\NBC)\cong H(\Tot \mc{C}^{*,*})\cong H^*(M^G;R)$.
  \end{proof}
  
  As an application of Theorem \ref{brokenspectral}, we get the following bounds on homological gradings for the cohomology rings of graph configuration spaces.
  
  \begin{corollary} 
    Let $G$ be a graph with $n$ vertices, and $R$ a commutative ring. Then $H^*(M^G,R)$ is supported in homological gradings $0\leq i\leq n-1$.
\end{corollary}

\begin{proof}
This follows directly from Theorem \ref{categcwhitney} and Theorem \ref{brokenspectral} since among subsets $S\subseteq 2^E$, the largest possible cardinality of a subset in $\NBC$ is $n-1$ as shown in the proof of Corollary \ref{support}.
\end{proof}

\bibliography{bib} 

\end{document}